\documentclass{amsart}

\usepackage{amsthm} 
\usepackage{color}

\usepackage{graphicx,epsfig}
\usepackage{epstopdf}
\usepackage{amsmath, amssymb, latexsym, euscript,amscd}
\usepackage{url}
\usepackage[all]{xy}
\usepackage{psfrag}
\usepackage{mathrsfs}

\usepackage{hyperref}
\hypersetup{
pdftoolbar=true,
pdfmenubar=true,
colorlinks=true,
linkcolor=black,
citecolor=black,
filecolor=black,
urlcolor=black
}

\newcounter{notes}%

\definecolor{darkgreen}{rgb}{0.0, 0.5, 0.0}

\newtheorem{theorem}{Theorem}[section]
\newtheorem{lemma}[theorem]{Lemma}
\newtheorem{corollary}[theorem]{Corollary} 
\newtheorem{definition}[theorem]{Definition}

\def\gap{\vspace{.3cm}\noindent}

\def\smallskip{\vspace{.15cm}}
\def\medskip{\vspace{.3cm}}
\def\text{\mbox}
\def\RR{{\mathbb R}}

\def\ZZ{{\mathbb Z}}
\def\PP{{\mathbb P}\,}
\def\HH{{\mathbb H}}

\def\P{{\mathbb P}}

\def\RP2{\operatorname{\mathbb{R}P}^2}
\def\RP3{\operatorname{\mathbb{R}P}^3}
\def\RP{\operatorname{\mathbb{R}P}}

\def\Fr{{\operatorname{Fr\;}}}

\def\interior{\operatorname{int}}
\def\SL{\operatorname{SL}}
\def\CH{\operatorname{CH}}

\def\SO{\operatorname{SO}}
\def\PO{\operatorname{PO}}
\def\PGL{\operatorname{PGL}}
\def\GL{\operatorname{GL}}

\def\Aff{\operatorname{Aff}}
\def\Hom{\operatorname{Hom}}

\def\cl{\operatorname{cl}}

\def\Aff{\operatorname{Aff}}
\def\tr{\operatorname{tr}}

\def\Diag{\operatorname{Diag}}

\def\C2{\operatorname{C^2}}
\def\halfgap{\vspace{.05in}}

\def\Fcal{\mathcal F}

\def\Image{\operatorname{Im}}
\def\Id{\operatorname{I}}

\def\NN{\mathbb N}

\def\O{\operatorname{O}}
\def\End{\operatorname{End}}

\def\UT{\operatorname{UT}}

\def\bdy{\partial}

\def\cl{\operatorname{cl}}

\def\stab{\operatorname{stab}}

\def\rank{{\bf r}}

  \def\UT{\operatorname{UT}}

\newcommand{\bv}{\left[\begin{array}{c}}
\newcommand{\ev}{\end{array}\right]}
\newcommand{\bbmat}{\begin{bmatrix}} 
\newcommand{\ebmat}{\end{bmatrix}}
\newcommand{\bmat}{\begin{matrix}} 
\newcommand{\emat}{\end{matrix}}
\newcommand{\bpmat}{\begin{pmatrix}} 
\newcommand{\epmat}{\end{pmatrix}}

\begin{document}
\title{On properly convex real-projective manifolds with Generalized Cusps}
\date{\today}
\begin{abstract} Suppose $E$ is an end of an irreducible, properly convex, real-projective $n$-manifold $M$. If 
$\pi_1E$ contains a subgroup of finite index isomorphic to  $\ZZ^{n-1}$,
 and $E\hookrightarrow M$ is $\pi_1$-injective,
then $E$  is a generalized cusp. We list some consequences when all ends are of this type.
Under certain hypotheses we prove the holonomy of a properly convex manifold is irreducible.
\end{abstract}

\noindent\address{DC: Department of Mathematics, University of California, Santa Barbara, CA 93106, UlSA}\\
\address{ST: School of Mathematics and Statistics, The University of Sydney, NSW 2006, Australia}
\address{}\\
\email{cooper@math.ucsb.edu}\\
\email{stephan.tillmann@sydney.edu.au}

\author{Daryl Cooper  and Stephan Tillmann}

\maketitle

A \emph{generalized cusp} is a properly convex, real-projective manifold $C$
that is diffeomorphic to $[0, \infty ) \times \partial C$ 
such that $\bdy C$ contains no line segment, and $\pi_1C$ is virtually nilpotent.

Generalized cusps were introduced in \cite{CLT2} and their theory developed in \cite{BCL1, BCL2}.
It was shown in \cite{BCL1} that if $C$ has compact boundary then 
the fundamental  group  is virtually abelian,
but see
\cite{C} and  \cite{CLT1}(5.9) for counter-examples when $\bdy C$ is not compact. In the rest of this paper 
the term {\em generalized cusp} will be only be used in the narrow sense that
 the boundary  is compact.

There is a growing literature concerning properly convex manifolds with ends of this type
\cite{BallasFig8, Ballastype2, BCL1, BCL2, BDL, BallasMarquis, choiends, Choibook, C, CLT1, CLT2, AL, Marq1}.

In applications it is
desirable to replace the geometric hypothesis on the boundary of a generalized cusp by an algebraic one.
This is done in Theorem~(\ref{virtab}),
 and is needed for forthcoming
work  by the authors \cite{CLT2}. 
 Theorem~(\ref{cuspends}) lists some consequences when all the ends are generalized cusps, and the fundamental group
 is relatively hyperbolic.

A  properly
convex set $\Omega\subset\RP^n$ is {\em reducible} if there
are disjoint proper subspaces $\RP^a,\RP^b\subsetneq\RP^n$ and 
 every point in $\Omega$ is contained in a line segment in $\Omega$
with one endpoint in  $\RP^a\cap \cl\Omega$ and the other in $\RP^b\cap\cl\Omega$.
Otherwise $\Omega$ is {\em irreducible}.
A properly convex manifold $M=\Omega/\Gamma$
is {\em irreducible} if $\Omega$ is irreducible.  Given $\Omega\subset\RP^n$ 
 recall that $\Fr\Omega=\cl(\Omega)\setminus\interior(\Omega)$ and $\bdy\Omega=\Omega\cap\Fr\Omega$.

A  non-compact 
submanifold $E\subset M$ is called an {\em end of $M$} if $E$ is the closure of  a component of $M\setminus\bdy E$, and for all $i\in\NN$
 there are compact submanifolds $K_i\subset K_{i+1}$ with $\bdy E\subset K_i\subset E$, and
 $E\setminus K_i$ is connected, and  $E=\cup K_i$. This is a special case of a more general definition
of {\em end} that suffices for our applications, and makes various statements simpler.
An end, $E$, of a properly convex manifold is called a {\em generalized cusp of $M$} if $E$ deformation retracts to a generalized cusp $C$.
 A subspace $A\subset B$ is {\em $\pi_1$-injective}
if, whenever a loop in $A$ is contractible in $B$, then the loop is contractible in $A$.

\begin{theorem}[irreducible implies generalized cusps]\label{virtab} Suppose $M$ is an irreducible properly convex $n$-manifold, and
$C$ is an end of $M$. If $C$ is $\pi_1$-injective, and $\pi_1C$ contains a subgroup of finite index isomorphic to $\ZZ^{n-1}$,
 then $C$ is a  generalized cusp of $M$.
\end{theorem}

Recall that a {\em radial flow} (\cite{CLT2} p.\thinspace 1384) is a one-parameter subgroup of $\PGL(n+1,\RR)$ such that the orbit of every point is 
a proper subset of a projective line, and
there is a point $c\in\RP^n$  called the {\em center  of the flow} 
that is in the closure of every flowline.
The stationary set is $H\cup\{c\}$, where $H\cong\RP^{n-1}$.
A {\em displacing hyperplane} (\cite{CLT2} p.\thinspace 1385) for a radial flow $\Phi$ is a hyperplane $P\ne H$ such that
$c\notin P$.  
 
       \gap
 
\emph{Sketch proof of~(\ref{virtab}).} We may assume $\pi_1C\cong\ZZ^{n-1}$, and 
  the holnomy of $C$ is a lattice $\Gamma$ in an abelian upper-triangular subgroup $T\subset\PGL(n+1,\RR)$. 
  Moreover there is
a {\em radial flow} $\Phi$ that centralizes $T$ and fixes a hyperplane $H$,
that is disjoint from $\Omega$. Furthermore the orbit of a generic point under $T\oplus\Phi$ is open. 
 The 
$T$-orbit of some point deep inside $\Omega$ is a convex hypersurface $S\subset\Omega$.
 Then 
 $C$ is a generalized cusp if and only if  $S$ contains no line segment. 
 If  $S$ is strictly convex we are done, otherwise
$S$ contains a maximal flat $F$. The subgroup $\stab(F)\subset T$ that preserves $F$ acts
 simply transitively on $F$. Hence $F$ is an open simplex, and there is a one-parameter
 subgroup $L\subset \stab(F)$ whose orbits in $F$ are line segments. Since $L$ commutes with $T\oplus\Phi$
  there is an open set of points whose $L$-orbits are line segments. This implies $L$
  preserves  $\Omega$, and hence $\Omega$ is reducible.\qed

A complete proof of Theorem~(\ref{virtab}) is given in Section~\ref{sec:Generalized cusps}. 
Section~\ref{sec:Irreducibility} provides the following:

\begin{theorem}[irreducible holonomy]\label{irred} Suppose $\Omega$ is properly
convex and $\Gamma\subset\PGL(\Omega)$ is finitely
generated, discrete and torsion free, and contains no non-trivial normal
abelian subgroup.
Suppose either $M=\Omega/\Gamma$ is  closed, or there is a subgroup  $G\cong\ZZ^{n-1}$ 
and $|\Gamma:G|=\infty$. Then  $\Gamma$ does not preserve any proper projective subspace. 
 \end{theorem}   
 
For example, this applies if $M$ is the interior of a compact $n$-manifold
that contains an embedded $\pi_1$-injective torus.
Theorems~(\ref{virtab}) and (\ref{irred})  imply:

\begin{theorem}[Generalized cusps are completely general]\label{mainthm}
Suppose $M$ is a properly convex $n$-manifold, and
$C$ is an end of $M$ and 
\begin{itemize}
\item $\pi_1C$ contains a subgroup of finite index isomorphic to $\ZZ^{n-1}$,
\item $C$ is $\pi_1$-injective,
\item $\pi_1M$ does not contain a non-trivial normal abelian subgroup,
\item $|\pi_1M:\pi_1C|=\infty$,
\end{itemize}
 then $C$ is a generalized cusp.
\end{theorem}

      \gap
  
  In particular this applies if $M$ is homeomorphic to a complete hyperbolic manifold of finite volume.
 Theorem~(\ref{cuspends}) in Section~\ref{sec:properties} gives some useful properties of manifolds whose ends are generalized cusps.

    \gap
\emph{Acknowledgements.}  The first author thanks the University of Sydney Mathematical Research
Institute (SMRI) for partial support and hospitality during completion of this work.
  Research of the second author is supported in part under the Australian Research Council's ARC Future Fellowship FT170100316.


\section{Generalized cusps are completely general}
 \label{sec:Generalized cusps}
 
 Let $W=\RR^{n+1}$.
If $\Omega\subset\PP W$, then $\PGL \Omega \subset\PGL W$ is the subgroup that preserves $\Omega$.
We write $\SL W$ for the subgroup of $\GL W$ with determinant $\pm1$, and
 $\SL\Omega$ is the preimage in $\SL W$ of $\PGL\Omega$.
 An element  $A\in\SL \Omega$ is {\em hyperbolic} if $A$ has an eigenvalue $\lambda$ with $|\lambda|\ne 1$,
and is {\em strongly hyperbolic} if, in addition, $A$ has unique eigenvalues of largest and smallest modulus, and they
are real, and
 have algebraic multiplicity one. It is {\em elliptic} if it is conjugate into $\O(n+1)$ and
 is {\em parabolic} if it is not hyperbolic and not elliptic.
 The same terms are applied to $[A]\in\PGL \Omega $. A subgroup  is {\em parabolic} if every element
 is parabolic or trivial.

 Suppose $M^n=\Omega/\Gamma$ is properly convex and $\Gamma\cong \ZZ^{n-1}$. We will
 show that $\Gamma$ is a lattice in a subgroup $T\cong\RR^{n-1}$ of $\PGL W$.
 Moreover the orbit of a generic point under $T$ is a convex hypersurface $S$, and $M$
 is a generalized cusp if and only if $S$ is strictly convex. In the remaining case $T$
 contains a one-parameter subgroup called a {\em linear flow} whose orbits are contained in lines.
 Moroever $\Omega$ is foliated
 by orbits and  is reducible.

A {\em linear flow} is an injective homomorphism $\Phi:\RR\to\PGL(n+1,\RR)$ such that the
orbit of every point in $\RP^n$ is  a proper subset of a projective line. If $\phi:\RR\to\RR$
is an isomorphism then $\Phi\circ\phi$ is called a {\em reparameterization} of $\Phi$.
There is a homomorphism $\Psi:\RR\to\GL(n+1,\RR)$ with  $\Phi=[\Psi]$, and
$\Psi_t=\exp(t M)$ for some matrix $M$. The eigenvalues
of $M$ correspond to weights for $\Psi$. If $\Psi'$ is another homomorphism
and $\Phi=[\Psi']$, then
$\Psi'_t=\exp(\lambda t)\Psi_t=\exp(t(M+\lambda I))$.
 This operation is called {\em rescaling}. We will abuse notation by also referring to $\Psi$ as a linear flow.

The {\em stationary subset} of $\Phi$
is the subset of $\RP^n$ consisting of all points that are fixed by the flow.
A linear flow is {\em parabolic} if, 
after reparamerization, there is $\pi\in\End(\RR^{n+1})$ with $\pi^2=0$ and $\Phi_t[x]=[x+t \pi(x)]$, in which case
the stationary subset is $\PP(\ker\pi)$.
It is {\em hyperbolic} if, 
after reparamerization,  there is a direct sum decomposition $\RR^n=A\oplus B$
and $\Phi_t[a+b]=[a+\exp(t)b]$, where $a\in A$ and $b\in B$; in which case the stationary subset is $\PP(A)\sqcup\PP(B)$.

\begin{lemma}   Every linear flow  is parabolic or hyperbolic. \end{lemma}
\begin{proof} Let $\Psi$ be a linear flow.
The hypothesis implies every point  $0\ne x\in\RR^{n+1}$ is contained
in some 2-dimensional subspace $V$ that is preserved by the flow. Moreover 
$\Psi|V$ has real weights, otherwise the orbit of $[x]$ is $\RP^1$.
Hence all the eigenvalues of $M$ are real. Suppose there are 3 distinct eigenvalues $\alpha_i$
corresponding to three points $[x_i]\in\RP^n$ each fixed by the flow.
Then there is an orbit $[\sum_i\exp(\alpha_i t)x_i]$ that is not contained in any $\RP^1$.
If there is only one eigenvalue, by rescaling $\Psi$, we may assume it is $0$.
If $M^2\ne0$ consideration of Jordan normal form contradicts that $\Psi$
is a linear flow. In this case $\Psi$ is parabolic.

This leaves the case there are exactly two eigenvalues $\alpha\ne\beta$.
 If there is a Jordan block for $\alpha$ of size bigger than $1$,
then we may assume $\alpha=0$ by rescaling, and there is a $2$-dimensional
subspace $V=\langle a,b\rangle$ with $\Psi_t(a+b)=a + t\cdot b$.
There is also $c\ne 0$ with $\Psi_t(c)=\exp(\beta t)c$. The  orbit of $a+b+c$ is
$[(a+tb)+\exp(\beta t)c]$ which is not contained in any $\RP^1$.
Thus $M$ is diagonalizable. By rescaling we may assume
one eigenvalue $\alpha=0$ and let $A$ be the eigenspace for $\exp \alpha$
and $B$ the other eigenspace. Then $\RR^n=A\oplus B$ and $\Psi_t$ is hyperbolic.
 \end{proof}
 
 \begin{lemma}\label{opensetorbit} Suppose $\Phi\subset\PGL(n,\RR)$ is a one parameter subgroup
 and $U\subset\RP^{n-1}$  is a non-empty open set such that the orbit under $\Phi$ of each point 
 in $U$ is a proper subset of a projective line. Then $\Phi$ is a linear flow.
\end{lemma}
\begin{proof} Let $\Phi=[\Psi].$ The set
$V$ consists of all triples $(a,b,c)$ with
$a,b,c\in\RR^{n}$ that are contained 
in some $2$-dimensional  linear subspace of $\RR^n$. Then $V$ is defined by the polynomial equations
given by setting the determinants of all  $3\times 3$ sub-matrices of $(a:b:c)$ equal to zero.
Let $W\subset\RR^n$ consist of all $a\in\RR^n$ such that the flow line containing
$[a]$ is contained is a projective line. Then $W$ equals the set of all $a$ such that
  $(a,\Psi_ta,\Psi_sa)\in V$ for all $s,t\in\RR$, and is therefore also a real algebraic variety. Since $W$
 contains the non-empty open set $U$ is follows that $W=\RR^n$ and therefore every
 orbit of $\Phi$ is a subset of a projective line. It only remains to show that no orbit is an entire
 projective line. If there is such an orbit then $\Phi$ has a weight that is not real.
 
 If $\Phi$ has a weight that is not real then
 there are
  $b_1,b_2\in\RR^n$ and $\gamma,\delta\in\RR$ with $\delta\ne0$ such that the orbit of $b_1$
  is $\Phi_t(b_1)=\exp(\gamma t)\left(\cos(\delta t) b_1+\sin(\delta t)b_2\right)$. 
  Choose $[a]\in U$. For all small $\epsilon$,
 then
  $c=a+\epsilon b\in U$. First assume the orbit of $[a]$ limits on two distinct points
  $[a_1],[a_2]\in\RP^{n-1}$. Then $a_1,a_2$ and $\Phi$ can be chosen
  so that $\Phi_t(a)= a_1+\exp(\alpha t)a_2$ with $\alpha\ne0$. Moreover
  $\langle a_1,a_2\rangle\cap\langle b_1,b_2\rangle = 0$. The four functions $1$, $\exp(\alpha t)$,
  $\exp(\gamma t)\cos(\delta t)$, $\exp(\gamma t)\sin(\delta t)$ are linearly independent.
  It follows that the orbit of $c$ contains four linearly independent vectors, which contradicts that
  the orbit of $[c]$ is contained in a line. 
 
  The remaining case is that $\Phi_t(a) = a + t d$. The four functions $1$, $t$,
  $\exp(\gamma t)\cos(\delta t)$, $\exp(\gamma t)\sin(\delta t)$ are linearly independent, 
  which is again a contradiction.
\end{proof}

 Suppose $\Phi\subset\PGL(n,\RR)$ is a $1$-parameter group.
A subgroup $\Gamma\subset\PGL(n,\RR)$ {\em approaches $\Phi$ at infinity}
if for every neighborhood $U\subset\PGL(n,\RR)$ of the identity, and every $s\in \RR$
there are $\gamma,\gamma'\in \Gamma$ and $t>s$ and $t'<-s$ such that 
$$\gamma\in U\cdot \Phi(t),\quad{\rm and}\quad
\gamma'\in  U\cdot \Phi(t')$$
In particular, if $\Gamma$ is a lattice in $T\cong\RR^m$, and $\Phi\subset T$ then $\Gamma$
approaches $\Phi$ at infinity. 

\begin{lemma}\label{approachlinear} Suppose $\Omega$ is properly convex, and $\Gamma\subset\PGL(\Omega)$
approaches a linear flow $\Phi$ at infinity. Then $\Omega$ is preserved by $\Phi$.
\end{lemma}
\begin{proof} Suppose 
 $p\in\Omega$  is not fixed by $\Phi$. Let $\ell=[a,b]$ be the closure of the flowline containing $p$. There
is sequence $\gamma_n\in\Gamma$ with
 $\gamma_n=\epsilon_n\circ\Phi(t_n)$ with $\epsilon_n\to \text{Id}$ and $t_n\to\infty$.
 Then $\Phi(t_n)(p)\to b$ so $\gamma_n(p)\to b$. Thus $b\in\cl\Omega$.
 Similarly $a\in\cl\Omega$ so $(a,b)\subset\Omega$.
\end{proof}

\begin{lemma}\label{preserveflow} If a closed  properly convex domain $\Omega$ is preserved by a linear flow 
$\Phi$, then $\Omega$ is reducible and $\Phi$ is hyperbolic. \end{lemma}
\begin{proof}
Given $p\in\interior\Omega$ the closure, $\ell$, of the flowline containing $p$
is contained in $\Omega$. If $\Phi$ is parabolic, then $\ell\cong\RP^1$. Thus $\Phi$ is hyperbolic
and $\ell=[a,b]$ has endpoints $a\in A\cap\Omega$ and $b\in B\cap\Omega$, where $A$ and 
$B$ are the stationary subsets
of $\Phi$. It follows that $\Omega$ is the convex hull of $A\cap\Omega$ and $B\cap\Omega$,
and it is therefore reducible. 
\end{proof}
The following is central to our approach. 

\begin{corollary}\label{approach} Suppose $M=\Omega/\Gamma$ is properly convex 
and $\Gamma$  approaches a linear flow at infinity. Then $M$ is  reducible. 
\end{corollary}
\begin{proof} By (\ref{approachlinear}) $\Omega$ is preserved by a linear flow, and by
(\ref{preserveflow})
$\Omega$ is reducible. 
\end{proof}

\begin{lemma}\label{stabflat} Suppose $\Omega$ is open, properly convex, and there is
an abelian group $T\subset\PGL(\Omega)$ that acts simply transitively on $\Omega$.
Then $\Omega$ is the interior of a simplex.
\end{lemma}
\begin{proof}
Since $\dim T>\dim\Fr\Omega$, for every $p\in\Fr\Omega$ 
there is $1\ne A_p \in T$ with $A_p(p)=p$.
Let $C_p$ be the component of $\Fr\Omega\cap \operatorname{Fix}(A_p)$ that contains $p$.
Then $C_p$ is a non-empty compact convex subset of $\Fr\Omega$.
Since $T$ is abelian and connected it follows that $C_p$ is preserved by  $T$.
Hence $\Fr\Omega$ is the union of $T$--invariant convex sets. 
Let $q\in\cl\Omega$ be an extreme point.
Then $q$ equals the intersection of those $C_p$ that contain it. Hence $q$ is fixed by all of $T$.
Every point in $\Omega$ is in the convex hull of an $n$-simplex, $\Delta$, with vertices that are extreme
 points of $\cl\Omega$. The vertices of $\Delta$ are fixed by $T$, therefore $T$ preserves the interior
 of $\Delta$. Since $T$ acts transitively on $\Omega$, and $\interior\Delta$ contains a point in $\Omega$,
 it follows that $\Omega=\interior(\Delta)$.
\end{proof}

\begin{definition} A {\em  quasi-cusp} is a properly convex $n$-manifold
$Q=\Omega/\Gamma$ such that  $\Gamma$ contains a
 finite index subgroup isomorphic to $\ZZ^{n-1}$.\end{definition}
Every cusp is a quasi-cusp. Observe that there is no requirement
on $\bdy\Omega$. 
If $Q$ is a quasi-cusp with boundary,
then $\interior Q$ is also a quasi-cusp.
The proof of Theorem~(\ref{virtab})
amounts to showing quasi-cusps are generalized cusps under some extra hypotheses. The definitions imply:

\begin{lemma} Suppose $Q$ is a quasi-cusp of dimension $n$. Then 
 $H^{n-1}(Q;\ZZ_2)\cong\ZZ_2.$
\end{lemma}

Suppose $C_i=\Omega_i/\Gamma_i$ are  generalised cusps for $i=1,2$  and $\dim C_i=n_i$.
Then there is a quasi-cusp $\Omega/\Gamma$, where $\Omega=\Omega_1*\Omega_2\subset\RP^{n_1+n_2+1}$
and $\Gamma=\Gamma_1\oplus\Gamma_2\oplus K$ and $K\cong\ZZ^2$ is  
a discrete subgroup of $\Phi_1\oplus\Phi_2\oplus\Phi\cong\RR^3,$ where $\Phi_i$ is a radial flow for $C_i$
and $\Phi$ is the linear flow that fixes each point in $\Omega_1$ and $\Omega_2$.

If $Q\cong\bdy Q\times [0,1)=\Omega/\Gamma$ is a quasi-cusp with compact boundary, there is a decomposition of  $\Fr\Omega$ into three parts
$$\Fr\Omega=\bdy\Omega\sqcup\operatorname{Fr}_v\Omega\sqcup\bdy_{\infty}\Omega$$
that is described below.  Moreover $\bdy_{\infty}\Omega=\RP^{n-1}_{\infty}\cap\cl\Omega$ and
 $\operatorname{Fr}_v\Omega$ is empty  for generalized cusps. 
\gap

\noindent{\em Example.} 
 In the following example the quasi-cusp is
   $Q=\Omega/\Gamma$,  with $\Omega\subset\RR^3$, and $\Gamma\subset\Aff(3)$.
   As an affine manifold $Q=S^1\times R$ where $S^1$ is the quotient of $(0,\infty)$ by a homothety, and $R$ is the quotient
   of the parabolic model $\{(y,z):2y\ge z^2\}$ of $\HH^2$ by a cyclic group of parabolics.

This can be referred to during the proof of Theorem~(\ref{prop}). 
{Define $\Omega=\interior(\Omega_M)$ where $\Omega_M\subset\RR^3$} and  $\tau:\RR^2\to \Aff(\Omega)$  are given by
$$\Omega_{M}=\{(x,y,z): 2y\ge z^2,\ x>0\},\qquad\qquad 
\tau(a,b)=\bpmat 
e^a & 0& 0 & 0\\
0 & 1 & b & b^2/2\\
0 &0 & 1 & b\\
0 & 0 & 0 & 1
\epmat$$
Then  $T=\tau(\ZZ^2)$ is a lattice in $T=\Image \tau$. Observe that  $\Omega$ is
properly convex and preserved by $T$,
so $Q=\Omega/\Gamma$ is a quasi-cusp.
Moreover, $T$ acts simply transitively on $\bdy\Omega_{M}$.
The radial flow $\Psi_t(x,y,z)=(x,y-t,z)$ is centralized by $T$
and $\Psi\oplus T$ acts simply transitively on $\Omega^+=\cup_t\Psi_t(\Omega)=\{(x,y,z): x>0\}$.
The plane $H$ where $y=-1$ is a displacing hyperplane for $\Psi$ that is disjoint from $\Omega$.

\begin{figure} \begin{center}
	 \includegraphics[scale=0.5]{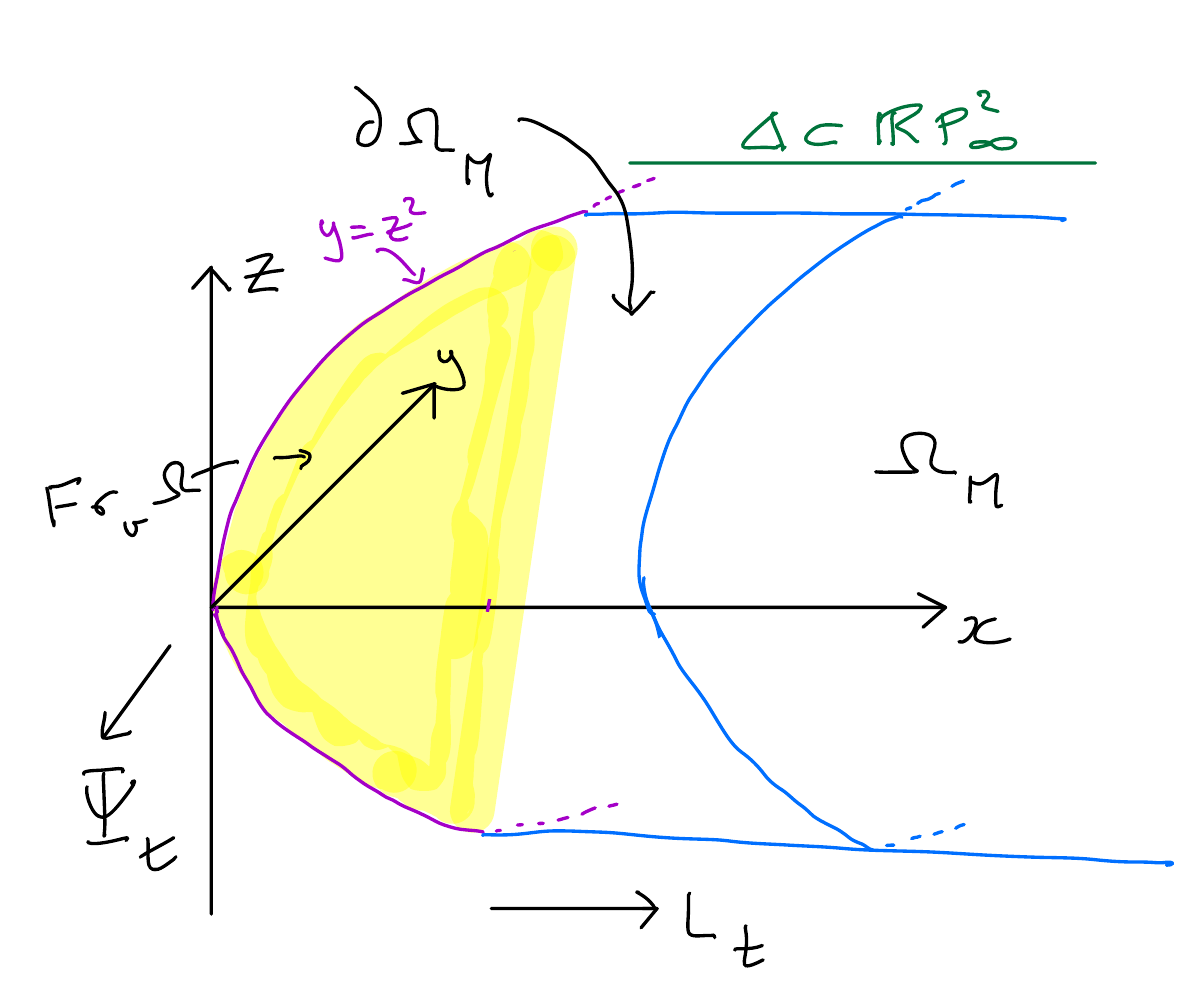}
	 \end{center}
 \caption{ $\Omega_M=\{(x,y,z): 2y\ge z^2,\ x>0\}$}\label{OmegaM}
\end{figure}

Also  $\Omega$ is
{\em backwards invariant,}  i.e. $\Psi_t(\Omega)\subset\Omega$ for all $t\le 0$. The surface
$$\bdy\Omega_{M}= \{(x,y,z):  2y=z^2,\ x>0\}$$ 
is   convex  and
  every point in $\bdy\Omega_M$ is contained in a straight line segment.
Moreover, $\bdy\Omega_{M}$ is transverse to the flow lines of $\Psi$ and is called 
the {\em flow boundary}. 
The {\em vertical frontier} 
$$\operatorname{Fr}_v\Omega:=(\Fr\Omega\setminus\bdy\Omega) \cap\RR^3=\{(x,y,z): 2y\ge z^2,\ x=0\}$$
is backwards invariant under $\Psi$.
The {\em ideal boundary}  is a  $1$-simplex  $$\bdy_{\infty}\Omega=\Fr\Omega\cap\RP^2_{\infty}=\{[x:1-x:0:0]: x\in[0,1]\}$$
The quasi-cusp $Q$ is foliated by convex ruled tori that are
 covered by level sets of $2y-z^2$.
These levels sets are permuted by $\Psi$.
Also $Q$ is a convex submanifold of $Q^+=\Omega^+/\Gamma\cong T^2\times\RR$ where
$\Omega+=\{(x,y,z): x>0\}\supset Q$. There is another
radial flow $L_t(x,y,z)=(\exp(t) x,y,z)$ that commutes with $T$. This flow preserves $\Omega$
and shows it is reducible: $\Omega_{M}$ is a cone  from $[1:0:0:0]$ to {$\operatorname{Fr}_v\Omega$}, 
which is a closed disk with one point deleted from the boundary.
Moreover $L_t$ is a subgroup of $T\oplus\Psi$.\qed
 
\gap
The subgroup $\UT(n)<\GL(n,\RR)$ consists of upper-triangular matrices with
positive diagonal entries. 
Let $\operatorname{D}(n)<\UT(n)$ be the subgroup of diagonal matrices.
Suppose $\Omega$ is properly convex and $K \subset \Omega$. The \emph{convex hull} $\CH(K)$ is the intersection
of all closed convex subsets of $\Omega$ that contains $K$.

\begin{theorem}\label{prop} Suppose $Q=\Omega/\Gamma$ is a quasi-cusp of dimension $n$.
Then either $Q$ contains a  generalized cusp or $\Omega$ is reducible.\end{theorem}
\begin{proof} 
We may assume $\Omega$ is open in $\RP^n$. 
By \cite{CLT2}(6.18)  there is a finite index subgroup $\Gamma_0\subset\Gamma$
such that  (after conjugacy) $\Gamma_0$ is a lattice in an 
upper-triangular connected nilpotent Lie subgroup $T\subset\UT(n+1)$. It follows from the definition
that if $Q$ has a finite cover  that is a generalized cusp then
$Q$ is a generalized cusp. Thus we may assume $\Gamma_0=\Gamma$.
By \cite{CLT2}(6.19) there is a radial flow $\Phi$ that centralizes
$T$.

Here is a sketch of the proof when the radial flow is parabolic. 
We want to show that the $T$-orbit of a point  $x$ deep inside $\Omega$ is a convex hypersurface inside $\Omega$.
Since $\Gamma$ is a lattice in $T$ and $\Gamma\cdot x\subset\Omega$ this seems reasonable.
Let $\Omega_t=\Phi_t(\Omega)$. To prove it, we enlarge $\Omega$ to be
the set $\Omega^+=\cup_t\Omega_t$ that is the union of flow lines thorough $\Omega$.
This
 is preserved by $\Gamma$ because $\Phi$ centralizes $\Gamma$. 
 Then $Q^+=\Omega^+/\Gamma$ is $Q$ plus an open collar. 
 A bit of work shows that $\Omega^+$ is preserved by all of $T$. If $x\in\Omega$ then $S'=T\cdot x$ 
 is a hypersurface in $\Omega^+$.  We do
 not yet know that $S'$ is convex. We
 use the fundamental-domains trick from \cite{CLT2}(6.24) to produce from $S'$ a convex codimension-1 submanifold 
 $S\subset\Omega_t$ (for some $t$) that is preserved by $T$. The action of $T$ on $S$ is simply transitive.
If $S$ is strictly convex we are done.  Otherwise there is a maximal
 flat $F\subset S$ and the stabilizer $H\subset T$ of $F$ acts simply transitively on $F$. 
There is a  subgroup $L$ of $H$ such that the $H$-orbit of every point in $F$ is a line segment. Since $L$ commutes
with $T\oplus\Phi$ it follows that the  orbit under $L$ of every point in $\Omega^+$ is a line segment, thus $L$
is a linear flow, and $\Omega$ is reducible. Now for the details.
\halfgap

 If $T\cap\Phi\ne 1$, then $\Phi\subset T$ so $\Gamma$ approaches $\Phi$ at infinity,
and this implies $\Omega$ is reducible.
Thus we may assume $T\cap \Phi=1$. Let $H$ be the stationary hyperplane and $c$ the center for $\Phi$. We may assume $H$ is disjoint
from $\Omega$, otherwise  replace $\Omega$ by one of the components, $\Omega'$, of $\Omega\setminus H$.
If $\Omega'/\Gamma$  contains a generalized cusp, we are done. If $\Omega'/\Gamma$ does not contain a generalized cusp, then below
we show that $\Gamma$
approaches a linear flow at infinity, and then it follows from~(\ref{approach}) that $\Omega$ is reducible.

The first case is that {$\Phi$ is parabolic, so $c\in H$. By \cite{CLT2}(6.19) we may choose $\Phi$ to be parabolic whenever $T$ is not diagonal.}
The affine patch $\RR^n=\RP^n\setminus H$ contains $\Omega$. Since $\Omega$ is properly convex there
is a displacing hyperplane $P\subset\RP^n$ that is disjoint from $\cl(\Omega)$.

After reversing the direction of the radial flow if needed, we may assume
$P_t:=\Phi_t(P)$ moves away from $\Omega$ as $t$ increases. 
After an affine change of
 coordinates we may assume  that $P$ is the hyperplane $x_1=0$ and {$x_1>0$ on $\Omega$}
and $\Phi_t(x)=x-t\cdot e_1$. 
Since $\Phi$ centralizes $T$, it centralizes $\Gamma$, so $\Gamma$ preserves
$\Omega_t=\Phi_t(\Omega)$. 

It follows from Claims 2, 3 and 4 of \cite{CLT2}(6.23) that ${\Omega_t}\subset\Omega$, whenever $t<0$, i.e. $\Omega$ is backwards invariant.
Moreover it follows that there is a properly convex set $\Omega_M\cong\bdy\Omega_M\times[0,\infty)$ 
with $\Omega\subset\Omega_M\subset\cl\Omega$, and that
 $Q=\interior(M)$, where $M=\Omega_M/\Gamma\cong \bdy M\times[0,\infty)$.
 Then $\Omega^+=\cup_t\Omega_t$ is the union of flowlines that contain a point of $\bdy\Omega_M$. Moreover $\Omega^+\subset\RR^n$
is convex and $F:\bdy\Omega_M\times\RR\to\Omega^+$ given by $F(x,t)=\Phi_t(x)$ is a homeomorphism. Also $\Gamma$
acts freely and properly discontinuously on $\Omega^+$ and  $\Omega^+/\Gamma\cong\bdy M\times\RR$.

We claim that $T$ preserves $\Omega^+$. Given $g\in T$, then $V(g):=\Omega^+\cap g(\Omega)^+$ is convex, open and 
$\Gamma$--invariant,
but it might be empty. Let $W\subset T$ be the set of all $g\in T$ such that $V(g)$ is not empty. Then $W$ is open. Moreover if $V(g)\ne\emptyset$
then $V(g)=\Omega^+$ because $V(g)/\Gamma$ is a convex submanifold of $\Omega^+/\Gamma$ and
it is a union of flowlines. Thus $V(g)=N\times\RR$ for some submanifold $N\subset\bdy M$. But $N$ and $M$ are $K(\Gamma,1)$'s
so $V(g)$ is a closed submanifold and therefore $N=\bdy M$. It follows that a neighborhood of the identity in $T$ preserves $\Omega^+$,
and therefore $T$ preserves $\Omega^+$

Following the proof of the first claim of \cite{CLT2}(6.24) there is a compact $X\subset T$ and compact $D\subset\bdy M$
so that $\Gamma\cdot D=\bdy M$ and $\Gamma\cdot X=T$. Let $\pi:\Omega^+\to\Omega^+/\Gamma$ be the projection and define
$$F:T\times {\bdy M}\to \Omega^+/\Gamma$$ 
by $F(g,x)=\pi(g\cdot x)$. Then $K=\Image(F)=F(X\times D)$ is compact. 
Thus there is $t$ such that $K\subset \Omega_t/\Gamma$. Hence $\bdy M=\Gamma\cdot K\subset\Omega_t$.
Now $\Omega_t$ is properly convex, so
$$Y=\cl(\CH(\pi^{-1}K))\subset\Omega_{t+1}$$ is properly convex, and $T$-invariant, and a closed subset of $\Omega_{t+1}$. Hence
${S}=\bdy Y$ is a convex hypersurface that is $T$-invariant.

Since $\dim T=\dim{S}$, and $T$ contains no elliptics, 
$T$ acts freely on ${S}$. This action is transitive since otherwise $(T\cdot x)/\Gamma$ is a $K(\Gamma,1)$ that is a proper submanifold of ${S}$, which is impossible. 
If ${S}$ is strictly convex
then $Y$ is a generalized cusp.

Otherwise there is a line segment in ${S}$. 
Hence ${S}$ contains a maximal flat $F$.  Then the subgroup $H\subset T$ that preserves $F$ acts simply transitively
on $F$. By~(\ref{stabflat}) $F\cong\interior(\Delta^k)$ is the interior of a simplex and $H|F$ is the projective diagonal group.
There is a $1$-parameter subgroup $L\subset H$ given by $L_t=\Diag(t,1,\cdots ,1)$ such that every orbit of $L$ in $F$ is a segment of a line.
Now $T\oplus\Phi$ acts simply transitively on $\Omega^+$.
Since $L$ commutes with $T$ and $\Phi$, it follows that the orbit under $L$ of every point in $\Omega^+$
is contained in a line. Then by~(\ref{opensetorbit}) $L$ is a linear flow. But $\Gamma$ approaches $L$ at infinity
so by~(\ref{approachlinear}) $\Omega$ is preserved by $L$. Then by~(\ref{preserveflow}) $\Omega$ is reducible.
{This completes the proof when $\Phi$ is parabolic.}

{If $\Phi$ is hyperbolic, then} $G=T\times\Phi$ is the diagonal group
in $\UT(n+1)$, and it follows that $\Omega^+$ is the interior of the $n$-simplex $\Delta$ with vertices 
$[e_1],\cdots,[e_{n+1}]$. Let $x\in\interior(\Delta)$ and consider the hypersurface $S=T\cdot x\subset\Delta$.
If $S$ is strictly convex then $C$ is a generalized cusp.
If $S$ is not convex  then $\Omega=\interior(\Delta)$ is reducible.
If $S$  is convex, but contains a flat, the argument above implies that $\Omega$ is reducible. 
\end{proof}


\section{Discreteness and Irreducibility}
\label{sec:Irreducibility}

This section shows that if $\pi_1M$ satisfies certain algebraic conditions then the holonomy of a properly
convex structure on $M$ is irreducible (\ref{irred}), and that a limit of such holonomies is always discrete and faithful
(\ref{noabelian}).

\begin{theorem}[Chuckrow's theorem \cite{CHU},\cite{KAP}(8.4)]\label{chuckrow} Suppose $\Gamma$ is a finitely
 generated group that does not contain a normal
infinite nilpotent subgroup $N$. Then the subset of $Hom(\Gamma,\GL(n,{\mathbb R}))$ consisting of discrete faithful
representations is closed in the usual (Euclidean) topology.
\end{theorem}

If $\Omega\subset\RP^n$, then $\PGL(\Omega)\subset\PGL(n+1,\RR)$ is the subgroup that preserves $\Omega$.
 We will make frequent use of the following implication.
\begin{corollary}\label{noabelian}
Suppose  $\Gamma$ is finitely generated and 
does not contain a non-trivial normal abelian subgroup. 
 Then the subset of $\Hom(\Gamma,\GL(n+1,{\mathbb R}))$ consisting of discrete faithful
representations is closed.
\end{corollary}
\begin{proof} 
Suppose $\Gamma$ contains an infinite normal nilpotent subgroup $G$.
Let $Z$ be the center of $G$. Then $Z$ is non-trivial and abelian.
Since $Z$ is characteristic in $G$, and $G$ is normal in $\Gamma$, it follows that $Z$ is normal  in $\Gamma$.
This  contradicts a hypothesis.  The result now follows from~(\ref{chuckrow}).\end{proof}
The next result, is due to Benoist \cite{Ben5}, see also \cite{MR3888622}. 
\begin{lemma}\label{virtabnormal}
If $M$ is closed and properly convex,
 then $\pi_1M$ contains a non-trivial normal abelian subgroup if and only if $\pi_1M$ has non-trivial virtual
center.
\end{lemma}
\begin{proof} Set $G=\pi_1M$. Suppose $Z=Z(H)\ne 1$ is the center of a finite index subgroup $H\subset G$. 
Since the universal cover of $M$
is contractible, $G$ is torsion-free, so $Z$ is infinite. There is a subgroup $H'\subset H$ of finite index
with $H'\vartriangleleft G$. Then $H'\cap Z$ has finite index in $Z$, and is central in $H'$,
  therefore  $Z'=Z(H')$ is non-trivial. Thus, after replacing $H$
by $H'$, we may assume $H\vartriangleleft G$. Now $Z$ is characteristic
in $H$ and thus normal in $G$. Hence $Z\ne 1$ is an infinite  normal abelian subgroup of $G$.

For the converse, suppose $1\ne A\vartriangleleft G$ and $A$ is abelian. 
Let $\Gamma$ be the image of the holonomy $\rho:\pi_1M\to\SL^{\pm}(n+1,\RR)$ of $M$. Every element of $\Gamma$ is hyperbolic
because $M$ is closed. Let $d_{\Omega}$ be the Hilbert metric on $\Omega$. The displacement function $\tau:\Gamma\to \RR$ is given by $$\tau(g)=\inf\{d_{\Omega}(x,gx):\ x\in\Omega\}$$
and $\mu=\min\tau(\pi_1M)>0$ because $M$ is compact. Since $A$ is abelian, the moduli of the weights of $\rho|A$ give a homomorphism
$$\lambda:A\to\RR_+^{n+1}$$
By (2.1) in \cite{CLT1} for $a\in A$ we have $\tau(a)=\log(\lambda_+/\lambda_-)$, where $\lambda_+,\lambda_-$ are the maximum and minimum
moduli of eigenvalues of $\rho(a)$. 
Since $\mu>0$ it follows that $\lambda$ is discrete and faithful. Hence
the subset $B\subset A$ of elements that minimize $\tau|A$ is finite.
Now $\tau(ghg^{-1})=\tau(h)$ so the action of $G$ on $A$ by conjugation permutes $B$. Since $B$ is finite, the kernel of this action
is a finite index subgroup $H\subset \pi_1M$ that fixes each element of $B$. Thus $B\subset Z(H)$.
\end{proof}

However when one deals with manifolds that are not closed, these statements are not equivalent, as the following example illustrates.

\halfgap

\noindent{\em Example.}\ \  Let $M$ be a 3-manifold that is a torus bundle over $S^1$ with monodromy $A\in\SL(2,\ZZ)$. Then $M$
is Euclidean if $A$ is periodic; NIL if $\tr(A)=\pm2$ and $A\ne\pm \text{Id}$; and otherwise $M$ has a SOLV geometry. 
For each $s>0$ there is an affine realization of this  structure as $\RR^3/\Gamma_s,$ where $\Gamma_s\cong\pi_1M\cong\ZZ\ltimes_A\ZZ^2$ is
the image of
 $$\rho_s(m,n,p)=\bpmat  A^p & 0 &  s\bpmat m\\ n\epmat\\
0 &1 & p\\
0 & 0 & 1
\epmat\qquad m,n,p\in\ZZ$$
This  gives a path of discrete  faithful
representations $\rho_s:\pi_1M\to\PGL(4,\RR)$ whose images converge
 to a cyclic group as $s\to0$. 
 
 Now $\PGL(4,\RR)$ acts on $\SL(4,\RR)/\SO(4)$
realized as the properly convex domain  $\Omega\subset\RP^5$ obtained by projectivizing the space of positive
definite quadratic forms on $\RR^3$. This gives a sequence of properly convex projective $5$-manifolds, homeomorphic to an
$\RR^2$-bundle over $M$, and the holonomies converge to a non-faithfull representation. In the SOLV case $\pi_1M$
has trivial virtual center, but contains $\ZZ^2$ as a normal subgroup.\qed

\gap

Let $W=\RR^{n+1}$ and suppose $\Omega\subset\PP W$ is open and $M=\Omega/\Gamma$ is a properly convex manifold.
We say {\em $\Gamma$ is reducible} if there is a proper projective subspace $P=\PP U\subset\PP W$
that is preserved by $\Gamma$. 
Let $W^*=\Hom(W,\RR)$ denote the dual vector space. The dual manifold $M^*=\Omega^*/\Gamma^*$
is properly convex and diffeomorphic to $M$. 
Now
$W=U\oplus V$ and $W^*=U^*\oplus V^*$ where $V^*=\{\phi\in W^*:\phi(U)=0\}$ and similarly for $U^*$.
Moreover $\Gamma$ preserves $\P(U)$ if and only if $\Gamma^*$  preserves $\P(V^*)$.

\begin{lemma}\label{redbdy} With the notation above,  if $\Omega'=P\cap  \Omega\ne\emptyset$
then $L=\Omega'/\Gamma$ is a convex submanifold of $M$  and $\dim L=\dim P$, and the inclusion 
$L\hookrightarrow M$  is a homotopy equivalence.

If $P\cap\cl\Omega=\emptyset$
then  $M$ contains a closed submanifold $L$ of codimension $\dim P$ and 
 $L\hookrightarrow M$  is a homotopy equivalence.
\end{lemma}
\begin{proof}
The first conclusion is immediate.  
Now suppose that $P\cap\cl\Omega=\emptyset$. There is a projective hyperplane $H$ that contains $P$
and is disjoint from $\cl\Omega$. Then $H=[\ker\phi]$ for some $\phi\in V^*$ with $\phi(U)=0$.
It follows that 
   $[\phi]\in W=\P(V^*)\cap\Omega^*\ne\emptyset$. The result now follows from the first part, using the diffeomorphism $M\cong M^*$.\end{proof}

\begin{lemma}\label{redsubgrp} With the hypotheses of (\ref{redbdy}) suppose either that $M$ is closed or else that
$\Gamma$ contains a subgroup $\Gamma'$
of infinite index and $H_{n-1}(\Omega/\Gamma';\ZZ_2)\ne 0$, then $\emptyset\ne P\cap\cl\Omega\subset\Fr\Omega$
\end{lemma}
\begin{proof} Otherwise by (\ref{redbdy}) there is a submanifold $L$ of $M$ such that the inclusion $L\hookrightarrow M$ is a homotopy
equivalence and $\dim L<\dim M$. 
If $M$ is closed then $H_n(M;\ZZ_2)\ne0$ but $\dim L<\dim M$ so 
 $H_n(L;\ZZ_2)=0$ which is a contradiction.  
 
 If $M$ is not closed, let $M'$ and $L'$ be the covers of $M$ and $L$ corresponding to $\Gamma'$.
Then $H_{n-1}(L';\ZZ_2)\cong H_{n-1}(\Omega/\Gamma';\ZZ_2)\ne 0$ and it follows that $L'$ is a closed manifold of dimension $(n-1)$, and therefore a finite cover of $L$. Hence $|\Gamma:\Gamma'|<\infty$, which
 contradicts $|\Gamma:\Gamma'|=\infty$.
\end{proof}
We will apply this when  $M$
contains a convex, closed  submanifold $N=\Omega'/\Gamma'$ of codimension one
with $\Omega'\subset\Omega$ and $|\Gamma:\Gamma'|=\infty$.
The following applies to a properly convex manifold that contains a generalized cusp.

\begin{lemma}\label{specialsubgrp} Suppose $M=\Omega/\Gamma$ is a properly convex manifold, and 
$\Omega'\subset\Omega$ is a convex closed subset bounded by a smooth, connected hypersurface
$\bdy\Omega'\ne\emptyset$, that contains no line segment. Let $\Gamma'\subset\Gamma$ be the subgroup  that preserves $\Omega'$
and suppose that $|\Gamma:\Gamma'|=\infty$, and $N=\bdy\Omega'/\Gamma'$ is a compact manifold. Suppose there is $\gamma\in\Gamma$ such that
$\gamma\Omega'\cap\Omega'=\emptyset$.
Then $\Gamma$ does not preserve a proper projective subspace.
\end{lemma}
\begin{proof} Suppose $\Gamma$ preserves $P=\PP(U)$. By (\ref{redsubgrp}) we may assume
that $W^+= \cl\Omega\cap P$ is not empty, convex, and is contained in $\Fr\Omega$.
Let $W=W^+\setminus\bdy W^+$, so that $W$ is an open properly convex set
and $\dim W<\dim M$.

First suppose there is $x\in W\cap\cl\Omega'$. Then $\gamma x\in W\cap\gamma\Omega'$ and
$d_{W}(x,x')<\infty$. Given $p\in\bdy \Omega'$ then, since $\Omega'$ is convex,
the segment $\ell=[p,x)$ is contained in $\Omega'$ and $\gamma\ell\subset\gamma\Omega'$. 
If $y$ is on $\ell$ and close enough to $x$ then $d_{\Omega}(y,\gamma y)\le d_W(x,\gamma x)+1$.
This implies $d_{\Omega}(y,\bdy\Omega')\le d_W(x,\gamma x)+1$. But $d_{\Omega}(y,\bdy\Omega')\to\infty$
as $y\to x$. It follows that $W\cap\cl\Omega'=\emptyset$.

Since $\Omega'$ is closed in $\Omega$, for each $x\in\Omega$ there is a point $y\in\Omega'$ 
that minimizes $d_{\Omega}(x,y)$. Since $\Omega'$ is convex, and $\bdy\Omega'$ smooth, and contains
 no line segment, 
$y$ is unique and the map $\pi:\Omega\rightarrow\Omega'$ given by $\pi x=y$ is distance non-increasing.
Moreover, if $x\in\bdy\Omega'$ then $\pi^{-1}x$ has closure a segment $[x,y]$ with $y\in\bdy\Omega$. Thus
there is a continuous extension $\pi:\cl\Omega\rightarrow\cl\Omega'$ where the closures are in projective space.
Let $W'=\Fr\Omega\setminus \Fr\Omega'$, then $\pi|:W\rightarrow\bdy\Omega'$
is a homeomorphism.

Clearly $\pi$ is $\Gamma'$ equivariant. Restricting gives an injective map $\pi :W\to\bdy\Omega'$.
 The action of $\Gamma$ on $\Omega'$ is free and properly
discontinuous. Thus the same is true  for the action on
$W$, and $\pi$ covers a map $f:W/\Gamma'\to\Omega'/\Gamma'=N$
that is a homotopy equivalence. Since $N$ is closed  it follows that $f$
is surjective and thus $\pi$ is a surjective. 
 Hence $W=W'=\Fr\Omega\setminus \Fr\Omega'$. But the same is true when $\Omega$
is replaced by $\gamma\Omega$. Thus there is $x\in\Fr\Omega'\cap\Fr\Omega=\Fr\Omega'\cap\Fr\gamma\Omega$.
As before there is a line $\ell=[p,x)$ in $\Omega$ and another line $\ell'=[p',x)$
in $\Omega'$ and this is a contradiction.
\end{proof}

The following restricts the fundamental group of a reducible manifold.
\begin{lemma}\label{plusZ} Suppose $\Omega$ is reducible and $M=\Omega/\Gamma$ is properly convex.
If $\Gamma$ contains no non-trivial normal abelian subgroup then there is a properly convex
manifold $\Omega/\Gamma^+$ with $\Gamma^+\cong\Gamma\times\ZZ$.
\end{lemma}
\begin{proof}  We have $\Omega=\Omega_U*\Omega_V$ with $\Omega_U\subset\PP U$ and 
$\Omega_V\subset \PP V$. Let $k=\dim U$ and $l=\dim V$.
Let $\rho:\pi_1M\rightarrow\GL U\oplus\GL V$ be the holonomy of $M$, so that
$\Gamma=\rho(\pi_1M)$.  Given $s\in\RR$ there is a homomorphism $\theta_s:\GL U\oplus\GL V\rightarrow\GL U\oplus\GL V$ given by 
$$\theta_s \bpmat A& 0 \\ 0 &B\epmat= \bpmat |\det A|^{-s/k} A& 0 \\ 0 & |\det B|^{-s/l}B\epmat$$
If $s\ne 1$ then $\theta_s$ is injective and has inverse $\theta_{1/(1-s)}$.
For $s\in[0,1]$ define $\rho_s=\theta_s\circ\rho$.

Observe that $\alpha I_U\oplus\beta I_V\in \GL(U)\oplus\GL(V)$ preserves $\Omega$. Hence
$\Gamma_s=\rho_s(\pi_1M)$ preserves $\Omega$. When for $s\ne 1$ then
$\theta_s:\Gamma\rightarrow\Gamma_s$ has an inverse,  so $\rho_s$ is discrete and faithful.
It follows from~(\ref{noabelian}) that  $\rho_1:\pi_1M\to\SL U\oplus\SL V$ is discrete faithful. 
Moreover  $\rho_1$ preserves $\Omega$, and $\Omega$ is properly convex, so
 this action is free and properly discontinuous.

Observe that $\Omega$ is preserved by the hyperbolic linear flow for $(U,V)$
given by  $\Phi:\RR\to\GL U\oplus\GL V$ where $\Phi(t)=\Id_U\oplus\exp(t)\Id_V$.
Let $$\tau:\pi_1M\oplus\ZZ\to\GL U\oplus\GL V\qquad\text{ given by}\quad \tau(\alpha,n)=\sigma(\alpha)\circ\Phi(n)$$
Then $\Gamma^+=\tau(\pi_1M\oplus\ZZ)$ preserves $\Omega$ and $\tau$ is discrete and faithful,
because $\det\tau(\alpha,n)=\exp(n)$. 
Thus $\Gamma^+$ acts freely and properly discontinuously on $\Omega$, so $R=\Omega/\Gamma^+$ is a properly convex manifold. 
\end{proof}

\begin{theorem}[Benoist]\label{closedirred} Suppose $M$ is a closed properly convex manifold and $\pi_1M$ has
trivial virtual center. Then the holonomy of $M$ is irreducible.
\end{theorem}
\begin{proof} Let $M=\Omega/\Gamma$.
 By (\ref{redsubgrp}) we may assume that $\emptyset\ne X= \PP(U)\cap\overline \Omega\subset\Fr\Omega$. 
Let $\Omega_U=X\setminus\bdy X$ be the relative interior of $X$. Thus $\Omega_U$ is properly
convex, and $\dim \Omega_U\le\dim\PP U$.
Now $W=U\oplus V$ and $\Gamma$ preserves $U$.
We may assume $U$ is chosen  to minimize $\dim V$.
The representation $\rho:\pi_1M\rightarrow \Gamma$
is given in matrix form by
$$\rho=\bpmat A & B \\ 0 & C\epmat$$
where $A:\pi_1M\to\GL(U)$ and $C:\pi_1M\to\GL(V)$.

We claim that replacing $B$ by $0$ gives a discrete faithful representation, 
$\rho_0=A\oplus C:\pi_1M\rightarrow\GL U\oplus\GL V$,
 that preserves a reducible properly
convex set $\Omega_0=\Omega_U*\Omega_V$. Then by~(\ref{plusZ}) there 
is a properly convex manifold $N=\Omega_0/\Gamma^+$. But $N$ is an $n$-manifold that is
homotopy equivalent to $M\times S^1$. The latter
 is a closed manifold of dimension $(n+1)$, and this is a contradiction. 

 It only remains to prove the claim. Observe that $A$ preserves the properly convex set $\Omega_U=\PP(U)\cap\overline \Omega\subset\PP(U)$.
The holonomy of the dual manifold $M^*$ is the dual representation
$\rho_0^*:\pi_1M\to\GL(W^*)$ which preserves $\PP V^*$.
By~(\ref{redsubgrp})  $\Omega'=\PP(V^*)\cap\overline\Omega^*$ is a 
non-empty subset of $\Fr\Omega^*$ and therefore properly convex.
 Also $\dim\Omega'=\dim \PP V^*$, otherwise
$\Omega'$ lies in a proper projective subspace of $ \PP V^*$ that is preserved
by $\rho^*(\pi_1M)$, and this contradicts minimality of $\dim V$.
Hence $\Omega_V:=(\Omega')^*\subset\PP(V)$ is a non-empty properly convex open set that is preserved by $C$.

For $0< t\le 1$ define 
$$P_t=\bpmat 1 & 0 \\ 0 & t^{-1}\epmat\qquad\text{and}\qquad \rho_t=\bpmat A & tB \\ 0 & C\epmat \ .$$
Observe that $\rho_1=\rho$ and $\rho_t$ is defined for $t=0$ and $\rho_0=A\oplus C$.
For $t>0$ notice that $\rho_t=P_t\rho_1P_t^{-1}$, so $\rho_t$ is discrete and faithful. 
Since $\rho_t\to\rho_0$ as $t\to 0$ and $\pi_1M$ has trivial virtual center, it follows that $\rho_0$ is discrete and faithful by~(\ref{noabelian}). 
 The action of $\rho$ and $\rho_0$ on $U$ are both equal to $A$
thus $\rho_0$ preserves $\Omega_U$.

 Now $\rho_0^*$ preserves $\Omega'$
and the action of $\rho_0$ and $\rho$ on $\PP V$ are both given by $C$ and are thus equal.
Hence $\rho_0$ preserves $\Omega_V$.

Thus $\Omega_0=\Omega_U*\Omega_V$ is properly convex and preserved by $\rho_0$. We claim that $\dim\Omega_0=\dim\Omega$. Let $W'\subset W$ be the vector subspace of minimal dimension such that
$\Omega'\subset\PP W'$. Then $\dim\Omega'=\dim\PP W'$ and $\rho_0$ preserves $\Omega_0$
and thus preserves $W_0$. Let $\rho_0':\pi_1M\rightarrow\GL(W')$ be the restricted action.
Replace $\Omega'$ be the  interior of $\Omega'$ in $W'$.

We claim $\rho_0'$ is discrete and faithful. 
Suppose $g\in\pi_1M$ and $\rho(g)=[L]$ is hyperbolic. Then there are $w_{\pm}\in W$ such that
 $[w_{\pm}]\in\Fr\Omega$
and $w_+$ is an attracting, and $w_-$ a repelling, fixedpoint of $[L]$. 
This means $Lv=\lambda_{\pm} v$ 
with $\lambda_{\pm}>0$ real
and $\lambda_{\pm}$ is the spectral radius of $L^{\pm1}$. Moreover the displacement distance of $\rho(g)$
for the Hilbert metric $d_{\Omega}$ is $\log\lambda_+/\lambda_-$.

Write $w_{\pm}=u_{\pm}+v_{\pm}$ with $u_{\pm}\in U$ and $v_{\pm}\in V$. 
Now $\rho_t$ preserves the properly convex domain $\Omega_t=P_t\Omega$ and 
$P_t(w_{\pm})=u_{\pm}+t^{-1}v_{\pm}$. If $v_{\pm}=0$ then $P_t(w_{\pm})=u_{\pm} $ gives a point $\Omega_u$.
If $v_{\pm}\ne 0$ then $\lim_{t\to0} P_t[w_{\pm}]=[v_{\pm}]$ is in $\Fr\Omega_V$. Thus in both cases
$\lim_{t\to0} P_t[w_{\pm}]$ is in $\Fr\Omega'$.  It follows that  $\rho$ and $\rho_0'$ have
the same the displacement distance. Since $M$ is compact there
is an element of $\pi_1M$ of shortest length. Hence $\rho_0'$ is discrete faithful.
\end{proof}

\begin{proof}[Proof of irreducible holonomy~(\ref{irred})]  If $M$ is closed this follows from~(\ref{closedirred}).
Otherwise there is a subgroup $G\cong\ZZ^{n-1}$ of $\Gamma$. If $\Omega/G$
is a generalized cusp the result follows from~(\ref{specialsubgrp}). Otherwise $\Omega$ is reducible
by~(\ref{prop}). Then by~(\ref{plusZ}) there is a properly convex manifold $P=\Omega/\Gamma^+$.
Now $\Gamma^+$ contains the subgroup $G^+=G\times\ZZ\cong\ZZ^n$. Then $N=\Omega/G^+$ is
an $n$-manifold with $\pi_1N\cong \ZZ^n$, so $H_n(N)\cong \ZZ$ and $N$ is closed. But $N$ covers
the manifold $P=\Omega/G^+$ so this covering is finite. This implies $|\Gamma^+:G^+|<\infty$.
This contradicts $|\Gamma:G|=|\Gamma^+:G^+|=\infty$ .
\end{proof}


\section{Properties of Generalized Cusps}
\label{sec:properties}

 If $\Omega$ is properly convex and $p\in\Fr\Omega$, then  a {\em supporting hyperplane at $p$} is a projective
 hyperplane $H$ such that $p\in H$ and $H\cap\interior\Omega=\emptyset$. The point $p$
  is a {\em smooth point} 
 if  $H$ is unique, and
 is a {\em strictly convex point} if there is 
  $H$ with $p=H\cap\cl\Omega$, and $p$ is a {\em round point} if it is both conditions hold.

If $\Omega$ is properly convex then a {\em properly embedded triangle} or {\em PET} in $\Omega$
 is a flat triangle $\Delta$ with $\interior\Delta\subset \interior\Omega$ and
$\bdy\Delta\subset\Fr\Omega$.
We say that the generalized cusp $C\cong\bdy C\times[0,1)$ is \emph{minimal size} if the only convex submanifold of $\cl C$ that contains $\bdy C$ is $\cl C$. This is always the case unless the holonomy is diagonalizable, see \cite{BCL1}(1.4)
It follows from \cite{BCL1}(1.24) that
\begin{lemma}\label{gencuspbdy} If $C$ is a minimal size generalized cusp of dimension $n$, then there is
$\Omega\subset\RR^n$ and $\Gamma\subset\Aff(\RR^n)$ such that $C=\Omega/\Gamma$.
Moreover $\bdy\Omega\subset \RR^n$ is a properly embedded, strictly-convex hypersurface and
$\Fr\Omega=\bdy\Omega\sqcup\Delta^{\rank}$, where $\Delta^{\rank}\subset\RP^{n-1}_{\infty}$
is a flat simplex called the {\em end flat}, and $0\le \rank\le n-1$ is the rank of $C$. In particular
$\Omega$ does not contain a PET.
\end{lemma}

If $\Omega$ is properly convex, then a {\em flat} in $\Fr\Omega:=\cl\Omega\setminus\interior\Omega$
 is a convex set that contains more than one point.  A flat is {\em maximal} if it is not a proper subset of another flat in $\Fr(\Omega)$.
Every flat is contained in at least one maximal flat.
Suppose $M$ is properly convex and $C\subset M$ is a generalized cusp. Let  $\pi:\Omega\to M$
be the projection and let $U\subset\Omega$ be a component of $\pi^{-1}C$.  Then 
$V=\Fr\Omega\cap\cl U=\bdy_{\infty} U$ is  
called {\em the end flat of $\Omega$ corresponding to $U$}.  By (\ref{gencuspbdy}) it is  
a flat simplex of dimension $\rank$.

 If $M$ is properly convex we say {\em all the ends of $M$ are generalized cusps} if there are pairwise disjoint
 $\pi_1$-injective generalized cusps
  $E_1,\cdots, E_n\subset M$ such that  $\cl(M\setminus\cup E_i)$ is compact.
 In this case we say {\em $\pi_1M$ is hyperbolic rel ends} if $\pi_1M$ is  hyperbolic 
  rel the subgroups $\{\pi_1E_i:1\le i\le n\}$ in the sense of Drutu \cite{drutu}.
   Note that it follows from the definitions that $\cl(M\setminus\cup E_i)$ is connected.
A subgroup of $\pi_1M$ is {\em an end group} if it is conjugate to some $\pi_1E_i$.

\begin{lemma}\label{disjointflats} Suppose $M=\Omega/\Gamma$ is properly convex and $V,V'$ are end flats of $M$
corresponding to $U,U'\subset\Omega$.
If $U$ and $U'$ are disjoint, then $V$ and $V'$ are disjoint.
\end{lemma}
\begin{proof} Suppose $x\in \cl(U)\cap\cl(U')\cap\Fr\Omega$. Generalized cusps are convex, therefore there
 are line segments
$\ell:[0,\infty)\to \cl U$  and $\ell':[0,\infty)\to \cl U'$ both parameterized
by arc length that both limit to $x$. Then $d(\ell(t),\ell(t'))$ is not increasing, and so is bounded above. Let $C=\pi U\subset M$
be the end covered by $U$.
Since $U$ and $U'$ are disjoint $d(\ell(t),\ell(t'))\ge d(\ell(t),\bdy U)$. But $d(\ell(t),\bdy U)=d(\pi\ell(t),\bdy C)\to\infty$
as $t$ increases, a contradiction.
\end{proof}

\begin{theorem}\label{endflats} Suppose $M=\Omega/\Gamma$ is properly convex with ends that
are generalized cusps, and $\pi_1M$ is hyperbolic rel the ends.
Then the end flats of $\Omega$ are pairwise disjoint, and
every flat in $\Fr\Omega$ is contained in an end flat. Moreover $\Omega$ does not contain a PET.
\end{theorem}
\begin{proof} 
Suppose $T$ is a PET in $\Omega$.
Then by \cite{drutu}, there is  $R>0$, and an end $E$ of $M$,  and a component $U$
of $\pi^{-1}E$  so that $T$ is contained
in the $R$--neighborhood, $W$, of $U$.   However $W$ is the universal cover of a generalized cusp.
By~(\ref{gencuspbdy}) $W$ does not contain a PET, hence $\Omega$ does not contain a PET.

The pairwise disjoint property follows from~(\ref{disjointflats}).
Let $B\subset M$ be compact such that the closure of each component of $M\setminus B$
is a generalized cusp.
Suppose $\ell\subset\Fr\Omega$ is a closed non-trivial line segment. 
Choose $p\in\interior\Omega$ and let
 $T\subset\cl\Omega$
be the triangle that is the convex hull of $p$ and $\ell$. Given $s>0$ let
$$T(s)=\{x\in\interior T\ :\ d_{\Omega}(x,\bdy T)\ge s\ \}\ .$$
Let $\pi:\Omega\to M$ be the projection. The first case is that for some $s>0$ the set $\pi(T(s))$
is disjoint from $B$. Then $\pi(T(s))$ is contained in some end $C$ of $M$. This implies $\ell$
is contained in an end flat corresponding to a component of $\pi^{-1}(C)$.

The remaining case is that there is a sequence $x_n\in T(n)$ such that $\pi x_n\in B$. 
Then there is a compact set $K\subset\Omega$ with $\pi(K)\supset B$.
There are $\gamma_n\in\Gamma$ such that $\gamma_nx_n\in K$, and we may subconverge so
that $x_n\to x_{\infty}$. We may also subconverge so that $\gamma_nT(n)$ converges to a PET
$T_{\infty}\subset\Omega$. By the above, this is impossible. Hence $\ell\subset\Fr\Omega$ always.
\end{proof}

\begin{lemma}\label{dual} Suppose $M=\Omega/\Gamma$ is properly convex and the ends are generalized cusps.
Then the dual manifold $M^*=\Omega^*/\Gamma^*$ has the same structure.\end{lemma}
\begin{proof} If $C$ is a generalized cusp  then it follows from the definition that $C^*$  is also a generalized cusp.
Suppose $C\subset M$ is a generalized cusp. Then $C^*\supset M^*$ and $C^*$ is a generalized cusp.
By the classification, \cite{BCL1}(0.2), $C^*$ contains a smaller generalized cusp that is contained in an end of $M^*$.
\end{proof}

\begin{theorem}[properties of generalized cusps]\label{cuspends} Suppose $M=\Omega/\Gamma$ is a properly convex manifold without boundary, 
and all the ends of $M$ are generalized cusps with compact boundary. Also suppose
 $\pi_1M$ is hyperbolic rel the ends, and $\pi_1M$ is not the union of the end groups.
  Let $\Fcal\subset\Fr\Omega$ be the union of the flats. Then 
\begin{itemize}
\item[1)] Maximal flats are pairwise disjoint.
\item[2)] Every maximal flat  is an end flat.
\item[3)] Every end flat is a maximal flat.
\item[4)] The stabilizer in $\Gamma$ of a maximal flat is an end group. 
\item[5)] Every parabolic subgroup is contained in an end group.
\item[6)] Every parabolic subgroup is  conjugate in $\PGL(n+1,\RR)$ into $\PO(n,1)$.
\item[7)] Every element of $\pi_1M$ is strongly hyperbolic or contained in an end group.
\item[8)] The set $X=\{x\in\Fr\Omega: \gamma(x)=x\ \text{ and } \gamma\text{ is strongly hyperbolic}\}$ consists of round points.
\item[9)]  $\Omega'=\interior(\CH X)$  is the unique minimal, non-empty, properly convex set preserved by $\Gamma$.
 \item[10)]  $X$ is dense in $\Fr(\Omega)\setminus\Fcal$.
 \item[11)] $\Omega'$ does not contain a PET.
 \item[12)] The dual manifold $M^*$ has the same properties.
\end{itemize}
\end{theorem}

\begin{proof}  Theorem~(\ref{endflats}) implies (1), (2), (3) and (11).
Let $V=\cl(\widetilde C)\cap\Fr\Omega$ be the end flat corresponding to the component $\widetilde C\subset\pi^{-1}(C)$ for a generalized cusp $C\subset M$.
If $\gamma\in\Gamma$ stabilizes $V$, then $\gamma(\widetilde C)=\widetilde C$, so $\gamma$ is an the endgroup of $C$. This proves (4).

If $\Gamma_0\subset\Gamma$ is a parabolic subgroup, then by \cite{CLT1}(4.7)
$\Gamma_0$ preserves a hyperplane $H$ and a point $p\in H\cap\Fr\Omega$. The only points in $\Fr\Omega$
that are fixed by a parabolic are points in the end flat. Hence the end flat for every element of $\Gamma_0$ is the same one.
Thus $\Gamma_0$ is conjugate into an endgroup, which proves (5). The classification of generalized cusps in \cite{BCL1}, with (5) implies (6). 

Suppose $\gamma\in\Gamma$  is hyperbolic.  The attracting
and repelling sets $S_{\pm}\subset\Fr\Omega$ of $\gamma$ are flat. If one of them is not a single point, then it is contained
in maximal flat, and thus an end flat. But this implies $\gamma$ preserves the end flat and is therefore in an endgroup.
Thus, if $\gamma$ is not in an end group, then the proof of \cite{CLT1}(2.8) now shows that $\gamma$ is strongly hyperbolic, which proves (7).

The set $X$
 is not empty because there is $\gamma\in\pi_1M$ that is not conjugate
into an end group, so by (7) $\gamma$ is strictly hyperbolic. 
Thus if $\gamma(x)=x$ then $x$ is an attracting or repelling fixed
point of $\gamma$ and is a strictly convex point.
 Moreover $x$ is a smooth point because the fixed points of the action of $\gamma$ on the dual domain
 are strictly convex points. This proves (8)

Since $X$ is preserved by $\Gamma$ it follows that  $\Omega'$
 is $\Gamma$-invariant, so $\dim\Omega'=n$ because $\rho$ is irreducible.
Clearly $\Omega'\subset\Omega$, so $\Omega'$ is properly convex, which proves (9).

Every point  $p\in\Fr\Omega'$  is in the limit of a sequence  $n$-simplices with vertices in $X$. If this limit is not a single
point $p$ then it is a flat that contains $p$. Hence $p$ is in a maximal flat. Otherwise $p$ is a limit of points in $X$, which proves (10).
Finally,~(\ref{dual}) implies (12).
\end{proof}

\small
\bibliography{Quasirefs.bib} 
\bibliographystyle{abbrv} 

\end{document}